\documentclass{amsart}
\usepackage{amsmath}
\usepackage{amssymb}
\usepackage{amsthm}
\usepackage{bbm}
\usepackage{cite}
\usepackage{leftidx}
\usepackage{pgf}
\usepackage{tikz-cd}
\usepackage{verbatim}
\usepackage{xypic}

\title{A note on commutative algebras and their modules in quasicategories}
\author{Saul Glasman}

\newcommand{\angs}[1]{\langle #1 \rangle}

\newcommand{\bb}{\mathbb}

\newcommand{\del}{\partial}
\newcommand{\D}{\Delta}

\newcommand{\Fun}{\text{Fun}}

\newcommand{\Hom}{\text{Hom}}

\newcommand{\inc}{\subseteq}

\newcommand{\iy}{\infty}

\newcommand{\mb}{\mathbf}
\newcommand{\mc}{\mathcal}
\newcommand{\mf}{\mathfrak}

\newcommand{\Mod}{\mb{Mod}}

\newcommand{\ol}{\overline}

\newcommand{\ot}{\otimes}

\newcommand{\Set}{\mc{S}et}

\newcommand{\wt}{\widetilde}
\newcommand{\X}{\times}

\newcommand{\CM}{\mc{CM}}
\newcommand{\F}{\mc{F}}
\newcommand{\Fo}{(\F_*)_{\angs{1}/}}

\theoremstyle{definition}

\newtheorem{cor}[section]{Corollary}
\newtheorem{dfn}[section]{Definition}

\newtheorem{lem}[section]{Lemma}
\newtheorem{prop}[section]{Proposition}

\newtheorem{thm}[section]{Theorem}
\begin{document}
\maketitle
The purpose of this document is to develop a neat combinatorial theory of modules over commutative algebras in $\iy$-categories in the vein of the theory of modules over associative algebras set out in \cite[\S 4.2, 4.3]{HA}. In fact, we'll describe a cute commutative analog $\CM^\ot$ of the operads $\mc{LM}^\ot$ and $\mc{BM}^\ot$ of \cite[\S 4.2.1]{HA} and \cite[\S 4.3.1]{HA}. This should ease the task of constructing and manipulating such modules. In particular, we prove the following theorem, which plays a role in \cite{Gla14}. It's tautological for 1-categories but sort of subtle for $\iy$-categories, and it's just generally nice to know:
\begin{thm}\label{fliy}
Let $\mb{C}^\ot$ be a symmetric monoidal $\iy$-category. We'll denote the category of finite sets by $\F$ and the category of finite pointed sets by $\F_*$. A datum comprising a commutative algebra $E$ in $\mb{C}$ and a module $M$ over it - that is, an object of $\Mod^{\F_*}(\mb{C})$, the underlying $\iy$-category of Lurie's $\iy$-operad $\Mod^{\F_*}(\mb{C})^\ot$ \cite[Definition 3.3.3.8]{HA} - gives rise functorially to a functor
\[A_{E, M} : \F_* \to \mb{C}\]
such that
\[A_{E, M}(S) \simeq E^{\ot S^o} \ot M.\]
\end{thm}
We'll prove Theorem \ref{fliy} by first describing the $\iy$-operad $\CM^\ot$ that parametrizes this data and then giving a map from $\F_*$ to the symmetric monoidal envelope of $\CM^\ot$. We'll work extensively with the model category of $\iy$-preoperads \cite[\S 2.1.4]{HA}, which we'll denote $\mc{PO}$. 
\begin{dfn}
We define the operad $\CM^\ot$ as (the nerve of) the 1-category in which
\begin{itemize}
\item an object is a pair $(S, U)$ consisting of an object $S$ of $\F_*$ and a subset $U$ of $S^o$;
\item a morphism from $(S, U)$ to $(T, V)$ is a morphism $f: S \to T$ in $\F_*$ such that for each $v \in V$, the set $U \cap f^{-1}(v)$ has cardinality exactly $1$.
\end{itemize}
\end{dfn}
It's easily checked that the functor $\CM^\ot \to \F_*$ that maps $(S, U)$ to $S$ makes $\CM^\ot$ into an $\iy$-operad. The following result isolates the hard work involved in proving Theorem \ref{fliy}:
\begin{prop} \label{anaya}
We give $\Fo$ the structure of an $\iy$-preoperad by letting the target map $t: \Fo \to \F_*$ create marked edges. Define a map of $\iy$-preoperads
\[\phi : \Fo \to \CM^\ot\]
by 
\[\phi(j : \angs{1} \to S) = \begin{cases} (S, \{j(1)\}) & \text{if }j(1) \in S^o \\
(S, \emptyset) & \text{otherwise.} \end{cases} \]
Thus we embed $\Fo$ as the full subcategory of $\CM^\ot$ spanned by those objects $(S, U)$ for which the cardinality of $U$ is at most 1. 

Then $\phi$ is an trivial cofibration in $\mc{PO}$.
\end{prop}

The proof will take the form of a series of lemmas.

\begin{lem} \label{payana}
Suppose $q : \mb{E} \to \mb{B}$ is an inner fibration of $\iy$-categories, $K$ a simplicial set and $r : K^\lhd  \to \mb{B}$ a map such that for each edge $e : k_1 \to k_2$ of $K$ and for each $l \in \mb{E}$ with $q(l) = r(k_1)$, there is a cocartesian lift of $e$ to $\mb{E}$ with source $l$. Denote the cone point of $K^\lhd$ by $c$ and suppose $d \in \mb{E}$ is such that $q(d) = r(c)$. Then there is a map $r' : K^\lhd \to \mb{E}$ lifting $r$ and taking every edge of $K^\lhd$ to a cocartesian edge of $\mb{E}$.
\end{lem}
\begin{proof}
Clearly we can lift in such a way that the image of every edge of $K^\lhd$ with source $c$ is cocartesian; let $r'$ be such a lift. We claim that $r'$ already has the desired property. Indeed, $r'$ can be viewed as a section of the cocartesian fibration $q_{K^\lhd} : \mb{E} \X_{\mb{B}} K^\lhd \to K^\lhd$, and then the result follows from \cite[Proposition 2.4.2.7]{HTT}.
\end{proof}

\begin{lem} \label{nauay}
Let $p: \mc{O}^\ot \to \F_*$ be any $\iy$-operad. For any set $T$, let $\mc{P}_T$ be the poset of subsets of $T$ ordered by reverse inclusion, and let $\mc{P}'_T = \mc{P}_T \setminus \{T\}$, so that $\mc{P}_T \cong (\mc{P}'_T)^\lhd$. For each $S \in \F_*$, let $r_S : \mc{P}_{S^o} \to \F_*$ denote the obvious diagram of inert morphisms. Let $X \in \mc{O}^\ot_S$, let $\rho : \mc{P}_{S^o} \to \mc{O}^\ot$ be the cocartesian lift of $r_S$ with $\rho(S^o) = X$ whose existence is guaranteed by Lemma \ref{payana}, and suppose $|S^o| > 1$. Then $\rho$ is a limit diagram relative to $p$.
\end{lem}

\begin{proof}
We work by induction on the size of $S^o$; the case $|S^o| = 2$ is an immediate consequence of the $\iy$-operad axioms. For each $k \in \bb{N}$, let $\mc{P}^{\leq k}_T$ denote the poset of subsets of $T$ of cardinality at most $k$. Then the restriction of $\rho$ to $(\mc{P}^{\leq 1}_{S^o})^\lhd$ is a $p$-limit diagram by the $\iy$-operad axioms. We now argue by induction on $k$ that for each $k$ with $1 \leq k \leq |S^o|-1$, the restriction of $\rho$ to $(\mc{P}^{\leq k}_{S^o})^\lhd$ is a limit diagram. Indeed, for each such $k > 1$, the $|S^0| = k$ edition of the lemma implies that $\rho|_{\mc{P}^{\leq k}_{S^o}}$ is $p$-right Kan extended from  $\rho|_{\mc{P}^{\leq k - 1}_{S^o}}$. Comparing the $p$-right Kan extensions along both paths across the commutative diagram
\[\begin{tikzcd}
\mc{P}^{\leq k  - 1 }_{S^o} \ar{r} \ar{d} & \mc{P}^{\leq k}_{S^o} \ar{d} \\
(\mc{P}^{\leq k - 1}_{S^o})^\lhd \ar{r} & (\mc{P}^{\leq k}_{S^o})^\lhd
\end{tikzcd}\]
gives the induction step, and thence the result.
\end{proof}

\begin{cor}\label{nayam}
Retaining the notation of Lemma \ref{nauay}: any nontrivial subcube of $\rho$ is a $p$-limit diagram. That is, if $U_1$ and $U_2$ are two subsets of $S^o$ with $U_1 \subseteq U_2$ and $|U_2 \setminus U_1| > 1$, then the restriction of $\rho$ to the subposet $\mc{P}_{U_1, U_2}$ spanned by those subsets $V$ of $S^o$ with $U_1 \subseteq V \subseteq U_2$ is a $p$-limit diagram.
\end{cor}
\begin{proof}
By restricting to $\mc{P}_{U_2}$, we may assume $U_2 = S^o$. Let
\[\mc{P}'_{U_1, U_2} = \mc{P}_{U_1, U_2} \setminus U_2\]
and let $\mc{Q}$ be the closure of $\mc{P}'_{U_1, U_2}$ under downward inclusion. Then $\mc{Q}$ contains $\mc{P}_{S^o}^{\leq 1}$, so by the above discussion, $\rho$ is $p$-right Kan extended from $\mc{Q}$. But $\mc{P}'_{U_1, U_2}$ is coinitial in $\mc{Q} = \mc{Q}_{U_2/}$, so we're good.
\end{proof}

\begin{proof}[Proof of Proposition \ref{anaya}]

Now let $\mc{O}^\ot$ be any $\iy$-operad, and let $F : \Fo \to \mc{O}^\ot$ be a morphism of $\iy$-preoperads. Consider the diagram
\[ \begin{tikzcd}
\Fo \ar{r}{F} \ar{d}{\phi} & \mc{O}^\ot \ar{d}{p} \\
\CM^\ot \ar{r} \ar[dotted]{ru}[below]{\phi_* F} & \F_*
\end{tikzcd} \]
We claim that the $\F_*$-relative right Kan extension $\phi_* F$ along the dotted line exists \cite[\S 4.3.2]{HTT}. Indeed, let $(S, U)$ be an object of $\CM^\ot$, and let $\mc{Q}_{(S, U)}$ be the subposet of $\mc{P}_{S^o}$ spanned by subsets $T$ such that
\begin{itemize}
\item $S^o \setminus U \subseteq T$, and
\item $|T \cap U| \leq 1$.
\end{itemize}
Then the natural map 
\[j_{(S, U)} : \mc{Q}_{(S, U)}^\lhd \to \CM^\ot\]
which takes all edges of $\mc{Q}_{(S, U)}^\lhd$ to inert edges of $\CM^\ot$ gives rise to a map
\[k_{(S, U)} : \mc{Q}_{(S, U)} \to \Fo \X_{\CM^\ot} (\CM^\ot)_{(S, U)/} \]
and $k_{(S, U)}$ is easily observed to be coinitial. Thus it suffices to show that $F \circ k_{(S, U)}$ admits a $p$-limit. But $F \circ k_{(S, U)}$ can be embedded, up to equivalence, into the cube of inert edges
\[\rho : \mc{P}_{S^o} \to \mc{O}^\ot\]
such that 
\[\rho_{S^o} = \leftidx{^p}{\prod_{U} F(\angs{1}, \{1\})} \X \leftidx{^p}{ \prod_{S^o \setminus U} F(\angs{1}, \emptyset)}\] 
where $\leftidx{^p}{\prod}$ denotes a product relative to $p$. By Corollary \ref{nayam} together with the induction argument used in the proof of Lemma \ref{nauay}, we see that $\rho_{S^o}$ is a $p$-limit of $F \circ k_{(S, U)}$. So $\phi_* F$ exists \cite[Lemma 4.3.2.13]{HTT}, and it is clear that $\phi_* F$ is a morphism of $\iy$-operads. Since any morphism of preoperads from $\Fo$ to an $\iy$-operad extends over $\CM^\ot$, $\phi$ must be a trivial cofibration.
\end{proof}
Now we'll relate our construction to Lurie's category of modules. Let $\mb{C}^\ot$ be an $\iy$-operad. Employing the notation of \cite[\S 3.3.3]{HA}, we define
\[\Mod(\mb{C}) := \Mod^{\F_*}(\mb{C})^\ot_{\angs{1}}\]
with analogous definitions of $\wt{\Mod}(\mb{C})$ and $\ol{\Mod}(\mb{C})$. 
\begin{prop}
There is an equivalence of $\iy$-categories
\[\Mod(\mb{C}) \simeq \Fun^\ot(\CM^\ot, \mb{C}^\ot).\]
\end{prop}
\begin{proof}
Let $X$ be a simplicial set equipped with the constant map $1_X : X \to \F_*$ with image $\angs{1}$. One then has set bijections
\[\Hom(X, \wt{\Mod}(\mb{C})) \cong \Hom_{\F_*}(X \X \Fo, \mb{C}^\ot)\]
and
\[\Hom(X, \ol{\Mod}(\mb{C})) \cong \Hom^\ot(X \X \Fo, \mb{C}^\ot)\]
where $\Hom^\ot$ denotes the set of $\iy$-preoperad maps; this is to say that we have an isomorphism of categories between  $\ol{\Mod}(\mb{C})$ and the category $\Fun^\ot(\Fo, \mb{C}^\ot)$ of $\iy$-preoperad maps from $\Fo$ to $\mb{C}^\ot$. Moreover, when $\mc{O}^\ot = \F_*$, the trivial Kan fibration $\theta$ of \cite[Lemma 3.3.3.3]{HA} is an isomorphism, so there is no difference between $\ol{\Mod}(\mb{C})$ and $\Mod(\mb{C})$.

Finally, we claim that the restriction map
\[\phi^* : \Fun^\ot(\CM^\ot, \mb{C}^\ot) \to \Fun^\ot(\Fo, \mb{C}^\ot)\]
is a trivial Kan fibration. Noting that the categorical pattern $\mf{P}_0$ of \cite[Lemma B.1.13]{HA} serves as a unit for the product of categorical patterns, we deduce from \cite[Remark B.2.5]{HA} that the product map
\[\mc{PO} \X \Set_\D^+ \to \mc{PO}\]
is a left Quillen bifunctor. This means, in particular, that for each $n$, the morphism of $\iy$-preoperads
\[(\Fo \X (\D^n)^\flat) \cup_{\Fo \X (\del \D^n)^\flat} (\CM^\ot \X (\del \D^n)^\flat) \to \CM^\ot \X (\D^n)^\flat\] 
is a trivial cofibration in $\mc{PO}$, which gives the result.
\end{proof}
We now wish to characterize the symmetric monoidal envelope of $\CM^\ot$.
\begin{dfn}
Let $\F^+$ be the category whose objects are pairs $(S, U)$, with $S$ a finite set and $U \inc S$ a subset, and in which a morphism from $(S, U) \to (T, V)$ is a morphism $f : S \to T$ of finite sets inducing a bijection $f|_U : U \cong V$. The disjoint union (which, mind you, is definitely not a coproduct) makes $\F^+$ into a symmetric monoidal category $(\F^+)^\amalg$.
\end{dfn}
\begin{prop}
The symmetric monoidal envelope $\text{Env}(\CM^\ot)$ is isomorphic to $(\F^+)^\amalg$.
\end{prop}
\begin{proof}
We briefly sketch the proof, which is a routine 1-categorical exercise. By definition, the symmetric monoidal envelope of $\text{Env}(\CM^\ot)$ has objects
\[(S, U, f: S^o \to T)\]
where $T \in \F$. Our isomorphism will map this object to the object
\[(T, (f^{-1}(t), f^{-1}(t) \cap U)_{t \in T})\]
of $(\F^+)^\amalg$.
\end{proof}
\begin{cor}
A $\CM^\ot$-algebra parametrizing a commutative monoid $E$ and a module $M$ in $\mb{C}^\ot$ gives rise to a functor $A_{E, M} : \F_* \to \mb{C}$ such that
\[A_{E, M}(S) \simeq E^{\ot S^o} \ot M.\]
\end{cor}
\begin{proof}
We construct $A_{E, M}$ by embedding $\F_*$ as the full subcategory of $\F^+$ consisting of objects $(S, U)$ for which $|U| = 1$.
\end{proof}

\bibliographystyle{alpha}
\bibliography{mybib}
\end{document}